\documentclass{jT}

\usepackage{amssymb,amsmath,latexsym,amsthm}
\usepackage[english]{babel}

\begin{document}

\newtheorem{teo}{Theorem}


\title[Dancs-He series for $\,\ln{2}\,$ and $\,\zeta{(2\,n+1)}$]{Using known zeta-series to derive the Dancs-He series for $\,\ln{2}\,$ and $\,\zeta{(2\,n+1)}$}

\author[F. M. S. {\sc Lima}]{{\sc F. M. S.} Lima}
\address{F. M. S. {\sc Lima}\\
Institute of Physics, University of Brasilia\\
P.O. Box 04455, 70919-970, Brasilia-DF, Brazil}
\email{fabio@fis.unb.br}


\date{\today}

\maketitle

\begin{resume}
Dans un ouvrage r\'{e}cent, Dancs et He ant trouv\'{e} une nouvelle formule d'Euler pour $\,\ln{2}\,$ et $\,\zeta{(2\,n+1)}$, $\,n\,$ \'{e}tant un entier positif, contenant chacune une s\'{e}rie qui ne peut apparemment pas \^{e}tre \'{e}valu\'{e}e dans une forme ferm\'{e}e, differement de $\,\zeta{(2\,n)}$, o\`{u} la formule d'Euler s'applique montrant que les m\^{e}mes valeurs de z\^{e}ta sont un multiple rationnel des $\,\pi^{2n}$. Il ya, toutefois, dans cette \'{e}tude, des formules issues de manipulations de certaines s\'{e}ries, en suivant la strat\'{e}gie des Tsumura, ce qui rend \emph{curieux} --- selon les termes de ces auteurs eux-m\^{e}mes --- l'apparition des nombres $\,\ln{2}\,$ et $\,\zeta{(2\,n+1)}$. Dans ce court article, je montre comment certaines z\^{e}ta-s\'{e}ries peuvent \^{e}tre utilis\'{e}es pour d\'{e}river les Dancs-He s\'{e}ries d'une mani\`{e}re plus directe.
\newline
\end{resume}

\begin{abstr}
In a recent work, Dancs and He found new `Euler-type' formulas for $\,\ln{2}\,$ and $\,\zeta{(2\,n+1)}$, $\,n\,$ being a positive integer, each containing a series that apparently can not be evaluated in closed form, distinctly from $\,\zeta{(2\,n)}$, for which the Euler's formula allows us to write it as a rational multiple of $\,\pi^{2n}$.  There in that work, however, the formulas are derived through certain series manipulations, by following Tsumura's strategy, which makes it \emph{curious} --- in the words of those authors themselves --- the appearance of the numbers $\,\ln{2}\,$ and $\,\zeta{(2\,n+1)}$.  In this short paper, I show how some known zeta-series can be used to derive the Dancs-He series in an alternative manner.
\newline
\end{abstr}



\bigskip

\section{Introduction}

The Riemann zeta function is defined, for real values of $\,s$, $s>1$, by\footnote{There is also a product representation due to Euler (1749), namely $\zeta(s) = \prod_{\,p}{{\,1/(1-p^{-s})}}$, where the product is taken over all prime numbers $p$, which is the main reason for the interest of number theorists in this function. As noted by Euler, since the harmonic series diverges, then from Eq.~(\ref{eq:Zdef}) one deduces that $\lim_{s \rightarrow 1^{+}}{\,\zeta(s)} = \infty$, which implies, from the product representation, that there is an infinitude of prime numbers.}
\begin{equation}
\zeta(s) := \sum_{k=1}^\infty{\frac{1}{k^s}} \: .
\label{eq:Zdef}
\end{equation}
For $\,s>1$, the series in Eq.~(\ref{eq:Zdef}) converges by the integral test and its sum for integer values of $s$ has attracted much interest since the times of J. Bernoulli, who proved that $\sum_{k=1}^\infty{{\,1/k^{\,2}}}$ converges to a number between $1$ and $2$.  Further, Euler (1735) proved that this sum evaluates to ${\,\pi^{2}/6}\,$, solving the so-called \emph{Basel problem}. He also studied this kind of series for greater integer values, finding, for even values of $\,s$, the notable formula (1750)
\begin{equation}
\zeta{(2\,n)} = (-1)^{n-1} \, \frac{2^{2 n-1}\,\pi^{2 n}}{(2 n)!} \,B_{2 n} ,
\label{eq:Euler}
\end{equation}
where $n$ is a positive integer and $\,B_{2 n}$ are Bernoulli numbers.\footnote{Since $B_{2 n} \in \mathbb{Q}$  and $\pi$ is a transcendental number, as first proved by Lindemann (1882), then Eq.~(\ref{eq:Euler}) implies that $\,\zeta{(2\,n)}\,$ is a transcendental number.}  For odd values of $s$, $s>1$, on the other hand, no analogous closed-form expression is known. In fact, not even an irrationality proof is presently known for $\,\zeta{(2\,n+1)}$, except for the Ap\'{e}ry proof that $\,\zeta(3)\,$ is irrational (1978)~\cite{Apery},  which makes the things enigmatic.

On trying to find out a closed-form expression for $\,\zeta{(2\,n+1)}\,$ similar to that in Eq.~(\ref{eq:Euler}), Dancs and He found a new ``Euler-type'' formula containing series involving the numbers $\,E_{\,2\,n+1}(1)$, where $\,E_{\,2\,n+1}(x)\,$ denotes the Euler's polynomial of degree $2\,n+1$~\cite{Dancs}.
~Their main result follows from some intricate series manipulations, in the lines of those found in Tsumura's proof of Eq.~(\ref{eq:Euler})~\cite{Tsumura}.  

However, the fortuitous appearance of the numbers $\,\ln{2}\,$ and $\,\zeta{(2n+1)}$ in the Dancs-He formulae, which is hard to be explained with usual series expansions, might well remain a mystery.  By noting that the numbers $\,E_{\,2\,n+1}(1)$ can be written in terms of $\,B_{2 n+2}\,$, and then in terms of $\zeta(2n)$, via Eq.~(\ref{eq:Euler}), I show here in this work how the Dancs-He series for $\,\ln{2}$  and $\,\zeta{(2n+1)}$ can be derived from some known zeta-series.
\newpage

\section{Dancs-He formula for $\,\ln{2}$}

For $\,\ln{2}$, Dancs and He found that (see Eq.~(2.6) of Ref.~\cite{Dancs}):

\begin{teo}[Dancs-He series for $\ln{2}$]  \label{teo:ln2}
\quad Let $E_{2 m+1}(x)$ denote the Euler's polynomial of degree $2 m+1$, $m$ being a nonnegative integer. Then
\begin{equation*}
\ln{2} = \frac{\pi^2}{2} \, \sum_{m=0}^\infty{(-1)^m \, \frac{\pi^{2 m}}{(2 m+3)!} \, E_{2 m+1}(1)} \, .
\label{eq:t1}
\end{equation*}
\end{teo}

\begin{proof}
$\,$~Let $L$ be the number at the left-hand side of Eq.~(\ref{eq:t1}).  By noting that $E_{2 m+1}(1) = -E_{2 m+1}(0) = 2 \, \frac{2^{2m+2}-1}{2m+2} \, B_{2m+2}$, one has
\begin{eqnarray*}
L &=& \pi^2 \sum_{m=0}^\infty{(-1)^m \frac{\pi^{2 m}}{(2 m+3)!} \left( 2^{2m+2}-1\right) \frac{B_{2m+2}}{2m+2} } \\
&=& -\sum_{m=0}^\infty{ (-1)^{m+1} \frac{B_{2m+2}}{(2m+2)\,(2m+3)\,(2 m+2)!} \left[ (2 \, \pi)^{2 m+2} - \pi^{2 m+2} \right] } \, .
\end{eqnarray*}
By substituting $n=m+1$, one finds that
\begin{eqnarray*}
L = -\sum_{n=1}^\infty{ (-1)^{n} \frac{B_{2n}}{2 n (2n+1) (2 n)!} \left[(2 \, \pi)^{2 n} - \pi^{2 n} \right] } = -\sum_{n=1}^\infty{ (-1)^{n} \, \frac{\pi^{2 n}\,B_{2n}}{(2 n)!} \, \frac{2^{2 n} - 1}{2 n (2n+1)\,} } \, .
\end{eqnarray*}
From Euler's formula for $\zeta(2n)$ in Eq.~(\ref{eq:Euler}), one has
\begin{equation}
L = \sum_{n=1}^\infty{ \left(1  - 2^{-2 n} \right) \, \frac{\zeta(2 n)}{n\,(2 n+1)} } \, .
\label{eq:ln2}
\end{equation}
Now, let us reduce this latter series to a simple closed-form expression.  For this, let us make use of the following series representation for $\zeta(s)$ introduced recently by Tyagi and Holm (see Eq.~(3.5) in Ref.~\cite{Tyagi}):
\begin{equation}
\frac{\zeta(s) \cdot \left(1-2^{1-s}\right)}{\pi^{s-1} \, \sin{({\pi s/2})}} = \sum_{n=1}^\infty{\left(2-2^{s-2n}\right) \, \frac{\Gamma(2n-s+1)}{\Gamma(2n+2)} \, \zeta(2n-s+1)} \, ,
\end{equation}
where $\Gamma(x)$ is the gamma function.\footnote{Note that $\Gamma(k) = (k-1)!$ for positive integer values of $k$.}  As the series in the right-hand side converges when we make $s=1$, all we need to do is to take the limit, as $s \rightarrow 1^{^{+}}$, of the factors at the left-hand side. The simple pole of $\zeta(s)$ at $s=1$ yields $\,\lim_{s \rightarrow 1^{^{+}}}{(s-1) \cdot \zeta(s)} = 1$.  Also, by applying the l'Hospital rule it is easy to show that $\lim_{s \rightarrow 1}{\,\left(1-2^{1-s}\right) / (s-1)} = \ln{2}$. The product of these two limits yields $\,\lim_{s \rightarrow 1^{^{+}}}{\zeta(s) \, \left(1-2^{1-s}\right)} = \ln{2}$, thus
\begin{equation}
\frac{\ln{2}}{\pi^0 \, \sin{(\pi/2)}} = 2\, \sum_{n=1}^\infty{\left(1-2^{-2n}\right) \, \frac{(2n-1)!}{(2n+1)!}} \, \zeta(2 n) \, ,
\end{equation}
which simplifies to
\begin{equation*}
\sum_{n=1}^\infty{\left(1-2^{-2n}\right) \, \frac{\zeta(2 n)}{\,n\,(2n+1)}} = \ln{2} \,.
\end{equation*}
From Eq.~(\ref{eq:ln2}), one has $\,L = \ln{2}$.  
\end{proof}
\vspace{1mm}

\section{Dancs-He formula for odd zeta-values}

Before presenting a general proof for the Dancs-He series for $\,\zeta(2n+1)$, $n$ being a positive integer, let us tackle the lowest case, i.e. $\zeta(3)$, a number for which several series representations have been derived since the times of Euler~\cite{LT2012}.  For this number, Dancs and He found the following series representation (see Eq.~(3.1) of Ref.~\cite{Dancs}).

\begin{teo}[Dancs-He series for $\zeta(3)$]
\label{teo:z3}
\begin{equation*}
\zeta(3) = \frac{\pi^2}{9} \, \ln{4} \, -\frac{2\,\pi^4}{3} \, \sum_{m=0}^\infty{(-1)^m \, \frac{\pi^{2 m}}{(2 m+5)!} \: E_{2 m+1}(1)} \, .
\label{eq:t2}
\end{equation*}
\end{teo}

\begin{proof}
$\,$~Let $S$ be the number for which the series at the left-hand side of Eq.~(\ref{eq:t2}) converges.  By substituting $E_{2 m+1}(1) = 2 \, \frac{2^{2m+2}-1}{2m+2} \, B_{2m+2}$ in this series, one has
\begin{eqnarray}
\pi^2 \, S &=& -2 \sum_{m=0}^\infty{(-1)^{m+1} \, \frac{\pi^{2 m+2}}{(2 m+5)!} \left( 2^{2m+2}-1\right) \, \frac{B_{2m+2}}{2m+2} } \nonumber \\
&=& -2 \sum_{m=0}^\infty{\frac{(-1)^{m+1} \, \left[(2 \, \pi)^{2 m+2}-\pi^{2 m+2}\right]}{(2m+2)\,(2m+3)\,(2m+4)\,(2m+5)} \: \frac{B_{2m+2}}{(2m+2)!}} \, .
\end{eqnarray}
By substituting $n=m+1$, one finds that
\begin{equation}
\pi^2 \, S = -2 \sum_{n=1}^\infty{\frac{(-1)^n \left[(2 \, \pi)^{2 n} -\pi^{2 n} \right]}{\,2 n\,(2n+1)\,(2n+2)\,(2n+3)} \, \frac{B_{2 n}}{(2 n)!}} \, .
\label{eq:juntos}
\end{equation}
From the relation between $B_{2 n}$ and $\zeta(2 n)$ in Eq.~(\ref{eq:Euler}), one has
\begin{equation}
\frac{\pi^2}{4} S = \sum_{n=1}^\infty{\frac{\zeta(2 n)}{2 n\,(2n+1)\,(2n+2)\,(2n+3)} } \, -\sum_{n=1}^\infty{\frac{\zeta(2 n)}{2 n\,(2n+1)\,(2n+2)\,(2n+3) \cdot 2^{2n}} } \, ,
\label{eq:z30}
\end{equation}
which is valid since the series in Eq.~(\ref{eq:juntos}) converges absolutely. The first series can be easily evaluated from a known summation formula (see Eq.~(713) in Ref.~\cite{LT2000}), namely
\begin{eqnarray}
\sum_{k=1}^\infty{\frac{\zeta(2 k)}{k\,(k+1)\,(2k+1)\,(2k+3)} \, t^{2k}} = \frac{\zeta(3)}{2 \pi^2} \, t^{-2} +\frac{\ln{(2 \pi)}}{3} -\frac{11}{18}  \nonumber \\
+ \frac{t^{-3}}{3} \left[ \zeta^{\,\prime}(-3,1+t)-\zeta^{\,\prime}(-3,1-t)\right] \, ,
\end{eqnarray}
where $\zeta(s,a)$ is the Hurwitz (or generalized) zeta function and $\zeta^{\,\prime}(s,a)$ is its derivative with respect to $s$.\footnote{The Hurwitz zeta function is classically defined for $\,\Re{(s)}>1\,$ as $\zeta(s,a) := \sum_{k=0}^\infty {\, 1 / (k+a)^s}$  ($a \ne 0,-1,-2,\ldots$), and its meromorphic continuation over the whole $s$-plane, with $\zeta(s,1) = \zeta(s)$, except by a simple pole at $\,s=1$.}  As this formula is valid for all $t$ with $|t| < 1$, it is legitimate to take the limit as $t \rightarrow 1^{-}$ on both sides, which yields
\begin{eqnarray}
\sum_{k=1}^\infty{\frac{\zeta(2 k)}{2 k\,(2 k+1)\,(2k+2)\,(2k+3)}} = \frac{\zeta(3)}{8 \pi^2} +\frac{\ln{(2 \pi)}}{12} -\frac{11}{72} \nonumber \\
+ \frac{1}{12} \, \lim_{\:t \rightarrow 1^{^{-}}}{\left[ \zeta^{\,\prime}(-3,1+t)-\zeta^{\,\prime}(-3,1-t)\right]} \, .
\label{eq:soma1}
\end{eqnarray}
The remaining limit is null, since $\,\lim_{t \rightarrow 1^{^{-}}}{\zeta^{\,\prime}(-3,1+t)} = \zeta^{\,\prime}(-3,2) = \zeta^{\,\prime}(-3) = \lim_{t \rightarrow 1^{^{-}}}{\,\zeta^{\,\prime}(-3,1-t)}$, which reduces Eq.~(\ref{eq:soma1}) to
\begin{equation}
\sum_{k=1}^\infty{\frac{\zeta(2 k)}{2 k\,(2 k+1)\,(2k+2)\,(2k+3)}} = \frac{\zeta(3)}{8 \pi^2} + \frac{\ln{(2 \pi)}}{12} \, -\frac{11}{72} \, .
\label{eq:s1}
\end{equation}
For the second series in Eq.~(\ref{eq:z30}), let us make use of the Wilton's formula (Eq.~(38) at p.~303, in Ref.~\cite{LT2012};  also Eq.~(54) in Ref.~\cite{Cho2006}):
\begin{equation}
\sum_{k=1}^\infty{\frac{\zeta(2 k)}{k\,(k+1)\,(2k+1)\,(2k+3) \cdot 2^{2k}}} = \frac{2 \zeta(3)}{\pi^2} + \frac{\ln{\pi}}{3} -\frac{11}{18} \, ,
\end{equation}
which can be written as
\begin{equation}
\sum_{k=1}^\infty{\frac{\zeta(2 k)}{2k\,(2k+1)\,(2k+2)\,(2k+3) \cdot 2^{2k}}} = \frac{\zeta(3)}{2 \pi^2} + \frac{\ln{\pi}}{12} -\frac{11}{72} \, .
\label{eq:s2}
\end{equation}
Now, by doing a member-to-member subtraction of Eqs.~(\ref{eq:s1}) and~(\ref{eq:s2}) and putting the result in Eq.~(\ref{eq:z30}), one has the desired result:
\begin{equation*}
\frac{2\,\pi^2}{9} \, \ln{2} \, -\frac{2\,\pi^4}{3} \, S = \zeta(3) .
\label{eq:fim}
\end{equation*}
\end{proof}


Now, let us generalize the above result for all $\,\zeta(2m+1)$, $m$ being a positive integer. See Eq.~(3.1) of Ref.~\cite{Dancs}.

\begin{teo}[Dancs-He series for odd zeta-values]  \label{teo:z2m1}
For any integer $m$, $m>0$,
\begin{eqnarray*}
\left(1-2^{-2m}\right) \, \zeta(2m+1) = \sum_{j=1}^{m-1}{\frac{(-1)^j\,\pi^{2j}}{(2j+1)!} \, (2^{2j-2m}-1) \, \zeta(2m-2j+1)}  \nonumber \\
-\frac{(-1)^m \, \pi^{2m} \ln{2}}{(2m+1)!} + \frac{(-1)^m \, \pi^{2m+2}}{2} \, \sum_{k=0}^\infty{(-1)^k \, \frac{\pi^{2 k} \, E_{2k+1}(1)}{(2k +2m +3)!} } \, .
\label{eq:t3}
\end{eqnarray*}
\end{teo}

\begin{proof}
$\,$~Let $\tilde{S}$ be the number for which the series at the right-hand side of Eq.~(\ref{eq:t3}) converges, i.e.
\begin{equation}
\tilde{S} := \sum_{k=0}^\infty{(-1)^k \, \frac{\pi^{2 k} \, E_{2k+1}(1)}{(2k +2m +3)!} } \, .
\end{equation}
By substituting $\,E_{2 k+1}(1) = 2 \, \frac{2^{2k+2}-1}{2k+2} \, B_{2k+2}\,$ in this series, one has
\begin{eqnarray}
\pi^2 \, \tilde{S} = \sum_{k=0}^\infty{(-1)^k \, \frac{\pi^{2 k+2}}{(2k+2m+3)!} \: 2 \left( 2^{2k+2}-1\right) \, \frac{B_{2k+2}}{2k+2} } \, .
\end{eqnarray}
By putting $n=k+1$, one finds that
\begin{equation}
\pi^2 \, \tilde{S} = -2 \sum_{n=1}^\infty{\frac{(-1)^n \pi^{2 n}}{(2 n+2m+1)!} \, \frac{\left(2^{2n}-1\right)}{2n} \, B_{2 n}} \, .
\end{equation}
From Eq.~(\ref{eq:Euler}), one has
\begin{eqnarray}
\frac{\pi^2}{2} \: \tilde{S} &=& 2 \sum_{n=1}^\infty{\left(1-2^{-2n}\right) \, \frac{(2n-1)!}{(2n+2m+1)!} \: \zeta(2 n) } \nonumber \\
&=& \sum_{n=1}^\infty{\left(2-2^{1-2n}\right) \, \frac{\Gamma(2n)}{\Gamma(2n+2m+2)} \: \zeta(2 n) } \, .
\end{eqnarray}
This latter series is just the one that appears in a formula for odd zeta-values derived recently by Milgran (see Eq.~(13) in Ref.~\cite{Milgran}), namely
\begin{eqnarray}
\zeta(2m+1) &=& \frac{(-1)^m \, \pi^{2m}}{1-2^{-2m}} \, \left[ -\frac{\ln{2}}{(2m+1)!} + \frac{\pi^2}{2} \, \tilde{S} \right]  +\frac{1}{1-2^{-2m}} \nonumber \\
&\times& \sum_{n=1}^{m-1}{\left(2^{2n-2m}-1\right) \, (-\pi^2)^n \, \frac{\zeta(2 m-2n+1)}{(2n+1)!} } \, .
\end{eqnarray}
By multiplying both sides by $1-2^{-2m}$, one has
\begin{eqnarray*}
\left(1-2^{-2m}\right)\,\zeta(2m+1) &=& (-1)^m \, \pi^{2m} \, \left[ -\frac{\ln{2}}{(2m+1)!} + \frac{\pi^2}{2} \, \tilde{S} \right] \nonumber \\
&+& \sum_{n=1}^{m-1}{(-1)^n\,\left(2^{2n-2m}-1\right) \, \pi^{2n} \, \frac{\zeta(2 m-2n+1)}{(2n+1)!} } 
\end{eqnarray*}
\begin{eqnarray*}
{\,} &=& -(-1)^m \, \pi^{2m} \, \frac{\ln{2}}{(2m+1)!} +(-1)^m \frac{\pi^{2m+2}}{2} \: \tilde{S} \nonumber \\
&+& \sum_{n=1}^{m-1}{\frac{(-1)^n\,\pi^{2n}}{(2n+1)!} \, \left(2^{2n-2m}-1\right) \, \zeta(2 m-2n+1) } \, ,
\end{eqnarray*}
which completes our proof.   
\end{proof}


\end{document}